\documentclass[a4paper, 11pt, reqno]{amsart}
\usepackage[numbers,sort&compress]{natbib}
\usepackage{amssymb, amsmath, mathrsfs, latexsym}
\usepackage[pdf]{pstricks}

\newcommand\GreenL{\mathcal{L}}
\newcommand\GreenR{\mathcal{R}}
\newcommand\GreenH{\mathcal{H}}
\newcommand\GreenD{\mathcal{D}}

\newcommand\trop{\mathbb{T}}

\newcommand\type{\operatorname{type}}

\newcommand\ft{\mathbb{FT}}

\newcommand\pft{\mathbb{PFT}}

\newtheorem{theorem}{Theorem}[section]
\newtheorem*{theorem11}{Theorem 1.1}
\newtheorem*{theorem14}{Theorem 1.4}
\newtheorem*{theorem15}{Theorem 1.5}
\newtheorem{lemma}[theorem]{Lemma}
\newtheorem{corollary}[theorem]{Corollary}

\newtheorem{proposition}[theorem]{Proposition}

\theoremstyle{definition}

\theoremstyle{definition}

\begin{document}

\noindent{This is the author accepted manuscript of an article
published in \\ 
Advances in Mathematics, Vol.~303 (2016), pp.~1236--1263. \\
\texttt{http://dx.doi.org/10.1016/j.aim.2016.08.033}.}

\vspace{2ex}

\noindent{\copyright\ 2016. This manuscript version is made available under the CC-BY-NC-ND 4.0 licence \texttt{http://creativecommons.org/licences/by-nc-nd/4.0/}}

\vspace{3ex}

\title[Pure Dimension and Projectivity of Tropical Polytopes]{Pure Dimension and Projectivity of \\ Tropical Polytopes}

\maketitle

\begin{center}

ZUR IZHAKIAN\footnote{
Institute of Mathematics, University of Aberdeen, AB24 3UE,
UK. Email \texttt{zzur@abdn.ac.uk}. Research
supported by Israel Science Foundation grant number 448/09 and
by a Leibniz Fellowship at the \textit{Mathematisches Forschungsinstut Oberwolfach}.
 Zur
Izhakian also gratefully acknowledges the support of EPSRC grant EP/H000801/1
and the hospitality of the University of Manchester during
a visit to Manchester.},
MARIANNE JOHNSON\footnote{School of Mathematics, University of Manchester,
Manchester M13 9PL, UK. Email \texttt{Marianne.Johnson@maths.manchester.ac.uk}. Research supported by EPSRC Grant EP/H000801/1.}
and
MARK KAMBITES\footnote{School of Mathematics, University of Manchester,
Manchester M13 9PL, England. Email \texttt{Mark.Kambites@manchester.ac.uk}. Research supported by an RCUK Academic Fellowship and EPSRC Grant EP/H000801/1.
Mark Kambites gratefully acknowledges the hospitality of the \textit{Mathematisches Forschungsinstut Oberwolfach}
during a visit to Oberwolfach.}

\date{\today}
\thanks{}

\end{center}

\date{\today}
\numberwithin{equation}{section}

\begin{abstract}
We study how geometric properties of tropical convex sets and polytopes,
which are of interest in many application areas, manifest themselves in
their algebraic structure as modules over the tropical semiring. Our main
results establish a close connection between \textit{pure dimension} of
tropical convex sets, and \textit{projectivity} (in the sense of ring
theory). These results lead to a geometric understanding of idempotency for
tropical matrices. 
As well as their direct
interest, our results suggest that there is substantial scope to apply
ideas and techniques from abstract algebra (in particular, ring theory)
in tropical geometry.
\end{abstract}

\section{Introduction}

Tropical mathematics can be loosely defined as the study of the real numbers
(sometimes augmented with $-\infty$) under the operations of
addition and maximum (or equivalently, minimum).
It has been an
active area of study in its own right since the 1970's
\cite{CuninghameGreen79} and also has well-documented applications in diverse
areas such as analysis of discrete event
systems, control theory, combinatorial optimisation and
scheduling problems \cite{MaxPlus95}, formal languages
and automata \cite{Pin98}, phylogenetics \cite{Eriksson05},
statistical inference \cite{Pachter04}, combinatorial/geometric
group theory \cite{Bieri84} and most recently in algebraic geometry
(see for example \cite{Itenberg09}). A key role in most of these areas is played by
\textit{tropically convex sets} and \textit{tropical polytopes}. These
subsets of tropical space are naturally endowed not only with a geometric
structure (as Euclidean polyhedral complexes), but also with a purely
algebraic structure (as modules over the tropical semiring).

In this paper we consider the way in which geometric properties of
convex sets and polytopes, of interest in many application areas, manifest
themselves in their algebraic structure. Our main results establish a close
connection between \textit{pure dimension} (a geometric property of interest for
applications of tropical methods) and \textit{projectivity} (in the sense
of ring theory and category theory). These results lead to a geometric
understanding of idempotency for tropical matrices. As a corollary, we obtain the
fact that all of the widely studied notions of
\textit{rank} for tropical matrices coincide where the matrices are idempotent or,
more generally, von Neumann regular. (This fact was mentioned without proof in
\cite[Fact 4, Section 35.7]{Akian06}, with reference given to a preprint of Cohen,
Gaubert and Quadrat which at the time of writing is still not available to the
public.)
As well as their direct interest,
these results suggest that there is substantial scope to apply ideas and techniques
from abstract algebra (in particular, ring theory) to understand problems
in tropical geometry.

We denote by $\ft$ the \textit{(finitary) tropical semiring}, which consists
of the real numbers under the operations of addition and maximum. We write
$a \oplus b$ to denote the maximum of $a$ and $b$, and $a \otimes b$ or just
$ab$ to denote the sum of $a$ and $b$. It is readily checked that both
operations are associative and commutative, $\otimes$ has a neutral element
($0$), admits inverses and distributes over $\oplus$, while $\oplus$ is
\textit{idempotent} ($a \oplus a = a$ for all $a$). These properties mean
that $\ft$ has the structure of an \textit{idempotent semifield (without
zero)}.

The space $\ft^n$ of tropical $n$-vectors admits natural operations of
componentwise maximum and the obvious scaling by $\ft$, which makes it
into an \textit{$\ft$-module}.
It also has a natural partial order. For
detailed definitions see Section~\ref{sec_prelim} below.
Submodules of $\ft^n$ play a vital role in tropical mathematics; as well
as their obvious algebraic importance, they have a geometric structure
in view of which they are usually called \textit{(tropical) convex sets}
or sometimes \textit{convex cones}. Particularly important are the finitely
generated convex sets, which are called \textit{(tropical) polytopes}.

There are several important notions of \textit{dimension} for convex sets.
The \textit{(affine) tropical dimension} is the topological dimension of the
set,
viewed as a subset of $\mathbb{R}^n$ with the usual topology. The
\textit{projective tropical dimension} (sometimes just called \textit{dimension}
in the algebraic geometry literature) is one less than the affine tropical
dimension. Note that,
in contrast to the classical (Euclidean) case, tropical convex sets may have regions
of different topological dimension. We say that a set $X$ has \textit{pure
(affine) dimension} $k$ if every open (within $X$ with the induced topology)
subset of $X$ has topological dimension $k$.
The \textit{generator dimension} (sometimes also called the \textit{weak dimension}) of a
convex set is the minimal cardinality of a generating subset, under the linear
operations of scaling and addition. The \textit{dual dimension} is the minimal
cardinality of a generating set under scaling and the induced operation of
\textit{greatest lower bound} within the convex set. We shall see later (Section~\ref{sec_dual}) that if a
convex set $X$ is the column space of a matrix, then its dual dimension is the generator
dimension of the row space, and also that the dual dimension of
$X$ is the minimum $k$ such that $X$ embeds linearly in $\ft^k$.

Since tropical convex sets also have the aspect of $\ft$-modules, it is
natural to ask about their algebraic structure. In particular, one might
ask whether important geometric and order-theoretic properties manifest
themselves in a natural way in their algebraic structure as modules,
and vice versa. If they do, this raises the twin possibilities of addressing
geometric problems involving polytopes by the use of (tropically linear)
algebraic methods and, conversely, using geometric intuition to provide
insight into problems in tropical linear algebra.

One of the most important properties in the study of modules is
\textit{projectivity}; recall that a module $P$ is called \textit{projective}
if every morphism from $P$ to another module $M$ factors through every
surjective module morphism onto $M$.
The main results of this paper characterise projectivity for tropical
polytopes, in terms of the geometric and order-theoretic
structure on these sets. Our most striking result gives a direct
connection between projective modules and the notions of dimension and
pure dimension discussed above:

\begin{theorem}\label{thm_geometric}
Let $X \subseteq \ft^n$ be a tropical polytope. Then $X$ is a projective $\ft$-module
if and only if it has pure dimension equal to its generator dimension and
its dual dimension.
\end{theorem}

Theorem~\ref{thm_geometric} will be proved at the end of Section~\ref{sec_geometry} below.
Recall that a square matrix $A$ is called \textit{von Neumann regular}
if there is a matrix $B$ such that $ABA = A$. In ring theory there is a well-established
three-way correspondence between von Neumann regularity, idempotency and projectivity:
for a finitely generated module, being projective is equivalent to being isomorphic to the
row space of an idempotent matrix, which in turn is equivalent to being isomorphic to the row space of
a von Neumann regular matrix. The corresponding result in the case of the tropical semiring with zero can be found in the work of Gaubert; indeed one can obtain an equivalent statement by combining Theorem 100 and Theorem 104 of \cite{Gaubert98}, the latter being based on joint work of Cohen, Gaubert and Quadrat \cite[Theorem 15]{Cohen97b} (which applies to a more general class of semifields). The same correspondence even applies when working over $\ft$, as a consequence of which Theorem~\ref{thm_geometric} immediately yields a new geometric characterisation of von Neumann regularity:


\begin{corollary}\label{cor_regular}
A square matrix over $\ft$ is von Neumann regular if and only if its row
space and column space have the same pure dimension equal to their
generator dimension.
\end{corollary}

Recall that a matrix $A$ is called \textit{idempotent} if $A^2 = A$.
Idempotent tropical matrices are of particular significance for metric
geometry, because of a natural relationship between the tropical
idempotency condition on a matrix and the triangle inequality for an
associated distance function (see for example \cite{Develin04} for more
details). A matrix is von Neumann regular if and only if it shares its column space (or
equivalently, its row space) with an idempotent matrix (see, for example, \cite[Theorem 104]{Gaubert98}), so Corollary~\ref{cor_regular} also exactly characterises the row and
column spaces of idempotent matrices. In fact, many of our results below
are proved by working with idempotents, and we believe the technical
understanding of tropical idempotency developed may prove to be of
independent interest.


Numerous definitions of \textit{rank} have been introduced and
studied for tropical matrices, mostly corresponding to different notions
of dimension of the row or column space.
For example, the \textit{tropical rank} of a matrix is the tropical
dimension of its row space, which by \cite[Theorem~23]{Develin04} for example, coincides
with that of its column space (variations of this result can also be found in \cite{Izhakian09} and \cite{Akian12}). The \textit{row rank} or \textit{row
generator rank} is the generator dimension of its row space, which we
shall see below (Section~\ref{sec_dual}) coincides with the dual dimension
of its column space. The \textit{column rank} or \textit{column generator
rank} is defined dually.
Other important notions of rank include
\textit{factor rank} (also known as \textit{Barvinok rank} or \textit{Schein rank}), \textit{Gondran-Minoux row rank}, \textit{Gondran-Minoux column rank},
\textit{determinantal rank} and
\textit{Kapranov rank}; since these do not play a key role in the present
paper we refer the interested reader to \cite{Akian06,Akian09,Develin05} for definitions.
Of these different ranks, none are ever lower than the tropical rank, and
none are ever higher than the greater of row rank and column rank; the
non-obvious parts of this claim are given in \cite[Remark~7.8 and Theorem~8.4]{Akian09}
and \cite[Theorem~1.4]{Develin05}. Thus, as a corollary we obtain a proof of the following
result (which was mentioned without proof in \cite[Fact 4, Section 35.7]{Akian06}):

\begin{corollary}\label{cor_equalrank}
Let $M$ be a square von Neumann regular matrix (for example, an idempotent
matrix) over $\ft$. Then the row
generator rank, column generator rank, tropical rank, factor/Barvinok/Schein rank,
Gondran-Minoux row rank, Gondran-Minoux column rank, determinantal rank and Kapranov rank of $M$ are all equal.
\end{corollary}

We also obtain an order-theoretic description of projectivity. For convex
sets whose generator dimension and dual dimension coincide with the dimension
of the ambient space (essentially, a non-singularity condition), this has
a particularly appealing form:
\begin{theorem}\label{thm_order1}
A tropical polytope in $\ft^n$ of generator dimension $n$ and dual dimension
$n$ is a projective $\ft$-module if and only if it is min-plus (as well as
max-plus) convex.
\end{theorem}
Theorem~\ref{thm_order1} can also be deduced from results mentioned in the abstract of
the talk \cite{CohenTalk07} but for which a proof has not yet been published.
In greater
generality the formulation is slightly more technical, but still quite straightforward:

\begin{theorem}\label{thm_order2}
A tropical polytope is projective if and only if it has
generator dimension equal to its dual dimension (equal to $k$, say), and is
linearly isomorphic to a submodule of $\ft^k$ that is min-plus convex (as
well as max-plus convex).
\end{theorem}

Theorems~\ref{thm_order1} and \ref{thm_order2} are established in
Section~\ref{sec_order} below. Polytopes that are min-plus (as well
as max-plus) convex have been studied in detail by Joswig and Kulas \cite{Joswig10},
who term them \textit{polytropes}. In the terminology of \cite{Joswig10},
a consequence of Theorem~\ref{thm_order1} is that a tropical $n$-polytope
in $\ft^n$ is a polytrope if and only if it is a projective $\ft$-module.
Theorem~\ref{thm_order2} says that a general tropical polytope is a
projective $\ft$-module if and only if it is linearly isomorphic to a
polytrope in some dimension.

As well as connecting algebraic and geometric aspects of polytopes, our
approach also yields further insight into the abstract algebraic structure
of the semigroup of all $n \times n$ tropical matrices, and in particular
the idempotent elements. For example, we
prove that any matrix of full column rank or row rank is $\GreenR$-related
(or $\GreenL$-related) to at most one idempotent (see Section~\ref{sec_prelim}
below for definitions and Theorem~\ref{thm_uniqueidpts} for the formal
statement and proof).

There are a number of different variants on the tropical semiring, which
arise naturally in different areas (for example algebraic geometry, traditional max-plus
algebra, and idempotent analysis). As well as the (theoretically trivial, but
potentially confusing) issue of whether to use maximum or minimum, one may choose
to augment $\ft$ with an additive zero element (see Section~\ref{sec_prelim}
below) or a ``top'' element (see for example \cite{Cohen04}).
In bridging different areas, and drawing on ideas and results from
all of them, we face the question of exactly which semiring to work in.
For simplicity, in this paper we have chosen to
establish most of our main results
only for $\ft$ as defined above; this choice seems to makes the ideas behind
the proofs clearest and also minimises the extent to which we must modify
and reprove existing geometric results to make them suitable for our needs. In places we are nevertheless forced to reprove some foundational results which are known for $\trop$-modules in the setting of $\ft$-modules; in other places it is more convenient to use known results over $\trop$ directly. We have tried to give detailed references for any results which are known in similar settings, but have included the proofs for the reader less familiar with the semiring literature. It is likely that our main geometric results can be extended to the augmented semirings themselves; the main modifications required would be the replacement of subtraction in the proofs by an appropriate notion of \textit{residuation} (see \cite{Blyth72}) and the extension of certain existing results which we rely upon (for example, those of \cite{Develin04}) to the new setting.

In addition to this introduction, this article comprises six sections.
Section~\ref{sec_prelim} briefly revises some necessary definitions,
while Section~\ref{sec_dual} introduces the key concept of the
\textit{dual dimension} of a polytope, and proves several equivalent
formulations. Section~\ref{sec_proj} proves some foundational results
connecting
projective modules, free modules and idempotent matrices over $\ft$ (some of which are already known in the case of modules over a semiring with 0 element, or specifically over $\trop$).
 Section~\ref{sec_order} and Section~\ref{sec_geometry} prove
our main results, giving order-theoretic and geometric characterisations
respectively of projective polytopes. Finally, Section~\ref{sec_examples}
presents some examples of how our concepts and results apply to tropical
polytopes in low dimension; while these are collected in one place for ease
of discussion, the reader may wish to refer to them at various times
throughout the paper.

\section{Preliminaries}\label{sec_prelim}

Recall that we denote by $\ft$ the set $\mathbb{R}$ equipped with the
operations of maximum (denoted by $\oplus$) and addition (denoted by
$\otimes$, or where more convenient by $+$ or simply by juxtaposition).
Thus, we write $a \oplus b = \max(a,b)$ and $a \otimes b = ab = a + b$.
Note that $0$ acts as a multiplicative identity.

We denote by $\trop$ the set $\ft \cup \lbrace -\infty \rbrace$ with
the operations $\oplus$ and $\otimes$ extended from the above so as
to make $-\infty$ an additive identity and a multiplicative zero, that
is
$$(-\infty) \oplus x = x \oplus (-\infty) = x \text{ and } (-\infty) x = x (-\infty) = -\infty$$
for all $x \in \trop$.
We also extend the usual order on $\mathbb{R}$ to $\trop$ in the obvious
way, namely by $-\infty \leq x$ for all $x \in \trop$.

By an \textit{$\ft$-module}\footnote{Some authors prefer the term
\textit{semimodule}, to emphasise the non-invertibility of addition, but
since no other kind of module is possible over $\ft$ we have preferred
the more concise term.}
 we mean a commutative semigroup $M$ (with operation $\oplus$)
equipped with a left action of $\ft$ such that $\lambda (\mu m) = (\lambda \mu) m$,
$(\lambda \oplus \mu) m = \lambda m \oplus \mu m$, $\lambda (m \oplus n) = \lambda m \oplus \lambda n$
and $0 m = m$
for all $\lambda, \mu \in \ft$ and $m, n \in M$.
A \textit{$\trop$-module} is a commutative monoid $M$ (with operation $\oplus$
and neutral element $0_M$) satisfying the above conditions with the additional
requirement that
$\lambda 0_M = 0_M = (-\infty) m$ for all $\lambda \in \trop$ and $m \in M$.
Note that the idempotency of addition in $\ft$ and $\trop$ forces the addition
in any $\ft$-module or $\trop$-module also to be idempotent.
There is an obvious notion of isomorphism between modules; we write $X \cong Y$
to indicate that two modules are isomorphic.

Let $R \in \lbrace \ft, \trop \rbrace$. We consider the space $R^n$ of $n$-tuples of $R$; if $x \in R^n$ then
we write $x_i$ for the $i$th component of $x$. Then $R^n$ admits an
addition and a scaling action of $R$ defined respectively by
$(x \oplus y)_i = x_i \oplus y_i$ and $(\lambda x)_i = \lambda (x_i) = \lambda + x_i$.
It is readily verified that these operations make $R^n$ into an $R$-module.
$R^n$ also admits a
partial order, given by $x \leq y$ if $x_i \leq y_i$ for all $i$, and a
corresponding componentwise \textit{minimum} operation, the minimum of two elements being
greatest lower bound with respect to the partial order.

In the case $R = \trop$
the vector $(-\infty, \dots, -\infty) \in \trop^n$ is an additive neutral
element
for $\trop^n$. The vector $(-\infty, \dots, -\infty, 0, -\infty, \dots, -\infty)$
with the $0$ in component $i$ is called the \textit{$i$th standard basis
vector} for $\trop^n$, and denoted $e_i$.

We write $M_n(R)$ for the set of all $n \times n$ square matrices over
$R$. Since $\oplus$ distributes over $\otimes$, these operations induce an
associative multiplication for matrices in the usual way, namely:
$$(AB)_{ij} = \bigoplus_{k = 1}^n A_{ik} \otimes B_{kj}$$
giving $M_n(R)$ the structure of a semigroup. (Of course one may equip it
with entrywise addition to form a (non-commutative) \textit{semiring},
but we shall not be concerned with this extra structure here.)
The semigroup $M_n(R)$ acts on the left and right of the space $R^n$ in
the obvious way, by viewing vectors as $n \times 1$ or $1 \times n$
matrices respectively.

A subset $X \subseteq R^n$
is called \textit{(max-plus) convex} if it is closed under $\oplus$ and the
action of $R$, that is, if it is an $R$-submodule of $R^n$. It is called
\textit{min-plus convex} if it is closed under
componentwise minimum and the action of $R$. A non-empty and
finitely generated
(under the linear operations of scaling and $\oplus$) submodule of $\ft^n$
is called a
\textit{(tropical) polytope}. Tropical polytopes in $\ft^n$ projectivize to compact
subsets of $\mathbb{R}^{n-1}$ with the usual topology \cite[Proposition~2.6]{Joswig05}.
Some examples of tropical polytopes are collected in
Section~\ref{sec_examples} at the end of this paper.

If $M$ is a matrix
over $R$ then its \textit{column space} $C_R(M)$ and \textit{row space}
$R_R(M)$ are the polytopes generated by its columns and its rows
respectively.

A non-zero element $x$ in a convex set $X \subseteq \trop^n$ or $X \subseteq \ft^n$ is called \textit{extremal} if for every
expression
$$x = \bigoplus_{i=1}^k y_i$$
with each $y_i \in X$ we have that $y_i = x$ for some $i$. Note that if
$x$ is extremal then $\lambda x$ is also extremal for all $\lambda \in \ft$.
It is immediate from the definition that if $X \subseteq \ft^n$ then its
extremal points do not depend upon whether it is considered as a subset of
$\ft^n$ or $\trop^n$.
It is known (see for example \cite{Butkovic07,Wagneur91}) that if $X$ is
finitely generated then it is
generated by its extremal points, and every generating set for $X$ contains
a scaling of every extremal point.

We shall also make use of some of \textit{Green's relations}, which are
a tool used in semigroup theory to describe the principal ideal structure
of a semigroup or monoid.
Let $S$ be any semigroup. If $S$ is a monoid,
we set $S^1 = S$, and otherwise
we denote by $S^1$ the monoid obtained by adjoining a new identity element
$1$ to $S$. We define binary relations $\GreenL$ and $\GreenR$ on $S$ by
$a \GreenL b$ if and only if $S^1 a = S^1 b$, and $a \GreenR b$ if and only if
$a S^1 = b S^1$. We define the binary relation $\GreenH$ on $S$ by
$a \GreenH b$ if and only if $a \GreenL b$ and $a \GreenR b$. Finally, the binary
relation $\GreenD$ is defined by $a \GreenD b$ if and only if
there exists an element $c \in S$ such that
$a \GreenR c$ and $c \GreenL a$. Each of $\GreenL$, $\GreenR$, $\GreenH$
and $\GreenD$ is an equivalence relation on $S$ \cite{Howie95}.

The following theorem, which first appeared in \cite{Gaubert98} over $\trop$ (see also \cite{K_tropd,K_tropicalgreen} for proofs over both $\trop$ and $\ft$) summarises some results
characterising Green's relations in the semigroups $M_n(\ft)$ and $M_n(\trop)$.
\begin{theorem}
\label{thm_green}
Let $A, B \in M_n(R)$ for $R \in \{\mathbb{FT},\mathbb{T} \}$.
\begin{itemize}
\item[(i)] $A \mathcal{L} B$ if and only if $R_R(A) = R_R(B)$;
\item[(ii)] $A \mathcal{R} B$ if and only if $C_R(A) = C_R(B)$;
\item[(iii)] $A \mathcal{D} B$ if and only if  $C_R(A)$ and $C_R(B)$ are linearly isomorphic;
\item[(iv)] $A \mathcal{D} B$ if and only if  $R_R(A)$ and $R_R(B)$ are linearly isomorphic.
\end{itemize}
\end{theorem}

We also need the following version of \cite[Theorem 104, part 3]{Gaubert98} over $\ft$, which can be seen as a consequence of \emph{tropical duality} (see for example \cite{Cohen04,Develin04,K_tropd}). This result follows immediately from Theorem~\ref{thm_green}(iii) and (iv) in the case of square matrices of the same size, but we shall also make use of it in the non-square, non-uniform case. We give a brief proof for expository purposes.

\begin{theorem}\label{thm_rowcolumn}
Let $M$ and $N$ be matrices over $\ft$ (not necessarily square or of the
same size). Then $C_\ft(M) \cong C_\ft(N)$ if and only if
$R_\ft(M) \cong R_\ft(N)$.
\end{theorem}
\begin{proof}
Suppose $f : C_\ft(M) \to C_\ft(N)$ is an isomorphism of $\ft$-modules.
By \cite[Theorem~2.4]{K_tropd} (see also \cite{Cohen04,Develin04}) there are \textit{anti-isomorphisms}
(bijections which invert scaling and reverse the partial order)
$g : R_\ft(M) \to C_\ft(M)$ and $h : C_\ft(N) \to R_\ft(N)$.
Then the composite $f \circ g : R_\ft(M) \to C_\ft(N)$ is clearly also
an anti-isomorphism, so by \cite[Lemma~2.3]{K_tropd}, the map
$h \circ f \circ g : R_\ft(M) \to R_\ft(N)$ is an isomorphism of $\ft$-modules.

The converse is dual.
\end{proof}

\section{Dual Dimension}\label{sec_dual}

Let $X \subseteq \ft^n$ be a convex set. We define the \textit{dual dimension}
of $X$ to be the minimum cardinality of a generating set for $X$ under the
operations of scaling and \textit{greatest lower bound} within $X$.
Beware that the greatest lower bound
operation within $X$ can differ from the componentwise minimum
operation, that is, the greatest lower bound operation in the
ambient space $\ft^n$. Indeed, they will coincide exactly if $X$ is min-plus
as well as max-plus convex.

The concept of dual dimension is in some sense
implicit in the theory of duality for tropical modules (see for example
\cite{Cohen04}), but to the authors' knowledge it was first explicitly
mentioned in \cite{K_tropj}, and has yet to be extensively studied.
Some examples of the dual dimension of polytopes are presented in
Section~\ref{sec_examples} below. We
next prove some alternative characterisations of dual dimension, which
we hope will convince the reader of its significance.

\begin{proposition}\label{prop_matrixdualrank}
Let $M$ be a (not necessarily square) matrix over $\ft$.
Then the dual dimension of $C_\ft(M)$ is
the generator dimension of $R_\ft(M)$ (that is, the row generator rank of the
matrix $M$).
\end{proposition}
\begin{proof}
It is known \cite[Theorem~2.4]{K_tropd} (see also \cite{Cohen04,Develin04}) that there is an
\textit{anti-isomorphism} (a bijection which inverts scaling and
reverses the order) from $R_\ft(M)$ to $C_\ft(M)$. This map takes
scalings to scalings, and maps the $\oplus$ operation in $R_\ft(M)$ to
greatest lower bound within $C_\ft(M)$. Thus, the generator dimension of $R_\ft(M)$
(the minimum number of generators for $R_\ft(M)$ under $\oplus$ and scaling)
is equal to the dual dimension of $C_\ft(M)$ (the minimum number of generators for $C_\ft(M)$
under greatest lower bound and scaling).
\end{proof}

\begin{theorem}\label{thm_dualrankembedding}
Let $X \subseteq \ft^n$ be a tropical polytope. Then the dual
dimension of $X$ is the smallest $k$ such that $X$ embeds linearly in
$\ft^k$. In particular, the dual dimension of $X$ cannot exceed $n$.
\end{theorem}
\begin{proof}
Suppose $X$ has generator dimension $q$ and dual dimension $k$.
Now $X$ is the column space of an $n \times q$ matrix $M$, and it follows
by Proposition~\ref{prop_matrixdualrank} that $R_\ft(M)$ has generator
dimension $k$; in particular, $k$ is finite. Thus, $R_\ft(M)$ has $k$ distinct (up to scaling) extremal
points and these must all occur as rows of $M$. Choose $k$ rows to represent
the extremal points, and discard the others to obtain a $k \times q$ matrix
$N$. Then $R_\ft(N) = R_\ft(M)$, so by Theorem~\ref{thm_rowcolumn},
$X = C_\ft(M)$ is isomorphic to $C_\ft(N) \subseteq \ft^k$.

Now suppose $X$ embeds linearly into $\ft^p$; we need to show that $k \leq p$.
The image of this
embedding is a
convex set of generator dimension $q$ in $\ft^p$, and so can be
expressed as the column space of a $p \times q$ matrix $N$ say. By
Proposition~\ref{prop_matrixdualrank}, the row rank of $N$ is the
dual dimension $k$ of $X$. But the size of $N$ means that this cannot
exceed $p$, so we have $k \leq p$ as required.
\end{proof}

\section{Projectivity, Free Modules, Idempotents and Regularity}\label{sec_proj}

In this section, we briefly discuss some properties of finitely generated
projective $\ft$-modules, and their relationship to idempotency and
von Neumann regularity of matrices. The corresponding relationship over rings
is well known, and was extended to cover distributive idempotent semifields (such as $\trop$) in
\cite{Cohen97b, Cohen06, Gaubert98}. Some of the results given there carry over to the case of
$\ft$-modules by making suitable modifications; where this is the case we give references to the original result.

To begin with we shall need a simple description of
free $\ft$-modules of finite rank. For any positive integer $k$, define
$$F_k = \trop^k \setminus \lbrace (-\infty, \dots, -\infty ) \rbrace.$$
Then $F_k$ is closed under addition and scaling by reals, and hence has the structure
of an $\ft$-module.

\begin{proposition}\label{prop_freemod}
$F_k$ is a free $\ft$-module on the subset $\lbrace e_1, \dots, e_k \rbrace$ of standard
basis vectors.
\end{proposition}
\begin{proof}
It follows from general results about semirings with zero (see for example
\cite[Proposition~17.12]{Golan10}) that $\trop^k$ is a free $\trop$-module of rank $k$, with
free basis $\lbrace e_1, \dots, e_k \rbrace$. We claim that $F_k$ is a free $\ft$-module with the
same basis. Suppose $M$ is an $\ft$-module and $f : \lbrace e_1, \dots, e_k \rbrace \to M$ is a function.
We may obtain from $M$ a $\trop$-module $M^0$ by adjoining a new element $0_M$, and defining
$0_M \oplus m = m \oplus 0_M = m$, $(-\infty) m = 0_M$ and $\lambda 0_M = 0_M$ for all $m \in M^0$
and $\lambda \in \trop$. Now by freeness of $\trop^k$, there is
a unique $\trop$-module morphism $g : \trop^k \to M^0$ extending $f$. It follows from the definition of
$M^0$ that $g$ maps elements of $F_k$ to elements of $M$, so it restricts to an $\ft$-module
morphism $h: F_k \to M$ extending $f$. Moreover, if $h' : F_k \to M$ were another such map, then it
would extend to a distinct morphism from $\trop^k$ to $M^0$ extending $f$, contradicting the
uniqueness of $g$.
\end{proof}

An abstract characterisation of the finitely generated projective modules over a distributive semifield (such as $\trop$), given in terms of direct factors\footnote{Specifically, it is shown that a homomorphism $B$ between free modules $\mathcal{U}$ and $\mathcal{X}$ is (von Neumann) regular if and only of there exists a homomorphism between free modules $\mathcal{X}$ and $\mathcal{Y}$ whose kernel congruence induces a well-defined projection that maps each equivalence class onto a unique representative lying in the image of $B$.}, can be found in \cite[Theorem 15]{Cohen06}. Since our focus is on the connection between projectivity and dimension, we consider the following formulation, parts of which are well known for semirings with zero (see for example \cite[Proposition~17.16]{Golan10}) but which needs slightly more work for $\ft$. Recall that a \textit{retraction} of an algebraic structure is an idempotent endomorphism; the image of a retraction is called a \textit{retract}.

\begin{theorem}\label{thm_projchar}
Let $X$ be a polytope of generator dimension $k$.
Then the following are equivalent:
\begin{itemize}
\item[(i)] $X$ is projective;
\item[(ii)] $X$ is isomorphic to a retract of the free $\ft$-module $F_k$;
\item[(iii)] $X$ is isomorphic to the column space of a $k \times k$ idempotent matrix over $\ft$;
\item[(iv)] $X$ is isomorphic to the column space of an idempotent square matrix over $\ft$.
\end{itemize}
\end{theorem}
\begin{proof}
Suppose (i) holds. Since $X$ is $k$-generated, Proposition~\ref{prop_freemod}
means that
 there is a surjective morphism $\pi : F_k \to X$.
We also have the identity morphism $\iota_M : X \to X$. By projectivity, there is a map
$\psi : X \to F_k$ such that $\pi \circ \psi = \iota_X$. But now $\psi \circ \pi : F_k \to F_k$
is a retraction with image isomorphic to $X$, so (ii) holds.

Now suppose (ii) holds, and let $\pi : F_k \to F_k$ be a retraction with image
isomorphic to $X$. For each standard basis vector $e_i$, define
$$x_i = \pi(e_i) \in F_k$$
and let $E$ be the matrix whose $i$th column is $x_i$. Now viewing $E$ as a
matrix over $\trop$, we
see that
$$E e_i = x_i$$
for each $e_i$. So the action of $E$ agrees with the action of $\pi$
on the standard basis vectors, and hence by linearity on the whole of $F_k$.
Since $E$ clearly also fixes the zero vector in $\trop^k$, this means that
$E$ represents an idempotent map on $\trop^k$, and hence is an idempotent
matrix.

It remains to show that $E \in M_k(\ft)$, that is, that $E$ has no
$-\infty$ entries. Suppose for a contradiction that $E_{ij} = -\infty$.
We claim that $E_{im} = -\infty$ for all $m$. Indeed, if we had $E_{im} \neq -\infty$
then there
would be no $\lambda \in \ft$ such that $\lambda x_m \leq x_i$, which clearly cannot
happen since $x_m$ and $x_i$ lie in $C_\ft(E)$ which is isomorphic to $X$, a subset of
$\ft^n$. Since $E$ is idempotent we have
$$x_i = \bigoplus_{p=1}^k E_{pi} x_p.$$
But since $E_{ii} = -\infty$ this writes $x_i$ as a linear combination
of the other columns. This means that $C_\ft(E)$ is generated by
$k-1$ vectors, which contradicts the assumption that $X$, which is isomorphic
to $C_\ft(E)$, has generator dimension $k$.

That (iii) implies (iv) is obvious.

Finally, suppose (iv) holds, and let $E \in M_m(\ft)$ be an
idempotent matrix with column space
isomorphic to $X$. Then $E$ viewed as a matrix over $\trop$ acts on
$\trop^m$ by left multiplication. Since $E$ does
not contain $-\infty$ it clearly cannot map a non-zero vector to zero, so
it induces an idempotent
map $\pi : F_m \to F_m$ with image $C_\ft(E)$.
Now suppose $A$ and $B$ are $\ft$-modules, $g : C_\ft(E) \to B$ is a morphism
and $f : A \to B$ is a
surjective morphism. By the surjectivity of $f$, for each standard basis vector $e_i$ of $F_m$
we may choose an element $a_i \in A$ such that $f(a_i) = g(\pi(e_i))$. Now since $F_m$ is free
by Proposition~\ref{prop_freemod},
there is a (unique) morphism $q: F_m \to A$ taking each $e_i$ to $a_i$. Now for each $i$ we have
$$f(q(e_i)) = f(a_i) = g(\pi(e_i)),$$
so by linearity, $f(q(x)) = g(\pi(x))$ for all $x \in F_m$. But then by idempotency of $\pi$,
$$f(q(\pi(x))) =  g(\pi(\pi(x))) = g(\pi(x))$$
for all $x \in F_m$. Thus, if we let $p$ be the restriction of $q$ to $\pi(F_m)$ then
we have $f \circ p = g$, as required to show that $\pi(F_m) = C_\ft(E)$ is
projective and hence $X$ is projective.
\end{proof}

Recall that an element $x$ of a semigroup (or semiring) is called
\textit{von Neumann regular}\footnote{In the literature of semigroup theory
such elements are usually just called ``regular''; we use the longer term
``von Neumann regular'' for disambiguation from other concepts of regularity
for tropical matrices (see for example \cite{Butkovic10}).}
if there exists an element $y$
satisfying $xyx = x$; thus a matrix is von Neumann regular as defined
above exactly if it is von Neumann regular in the containing full matrix
semigroup. It is a standard fact from semigroup theory (see
for example \cite{Howie95}) that an element is von Neumann regular exactly
if it is $\GreenD$-related (or equivalently, $\GreenL$-related or $\GreenR$-related)
to an idempotent.

The following result extends to $\ft$ a fact which is known
for a class of semirings including $\trop$ (\cite[Theorem 15]{Cohen97b},
\cite[Theorem 104]{Gaubert98} and \cite[Proposition 5]{Cohen06}).

\begin{proposition}\label{prop_regchar}
Let $A \in M_n(\ft)$. Then the following are equivalent:
\begin{itemize}
\item[(i)] $A$ is a von Neumann regular element of $M_n(\trop)$;
\item[(ii)] $A$ is a von Neumann regular of $M_n(\ft)$;
\item[(iii)] $C_\ft(A)$ is a projective $\ft$-module;
\item[(iv)] $R_\ft(A)$ is a projective $\ft$-module.
\end{itemize}
\end{proposition}
\begin{proof}
We prove the equivalence of (i), (ii) and (iii), the equivalence of (i), (ii) and (iv) being dual.

If (i) holds, then $A = ABA$ for some $B \in M_n(\trop)$. An easy calculation
shows that replacing any $-\infty$ entries in $B$ with sufficiently small finite values
yields a matrix $B' \in M_n(\ft)$ satisfying $A = AB'A$, so (ii) holds.

If (ii) holds then $A$ is von Neumann regular, so it is $\GreenR$-related
to an idempotent matrix in $E \in M_n(\ft)$. Now by Theorem~\ref{thm_green}
we have $C_\ft(A) = C_\ft(E)$. But $C_\ft(E)$ is projective by
Theorem~\ref{thm_projchar}, so (iii) holds.

Finally, suppose (iii) holds, so $C_\ft(A)$ is projective, and let $k$ be
the generator dimension of $C_\ft(A)$. Note that $k \leq n$, since
$C_\ft(A)$ is generated by the $n$ columns of $A$. By Theorem~\ref{thm_projchar} there is an idempotent matrix $E \in M_k(\ft)$ such that $C_\ft(E)$ is
isomorphic to $C_\ft(A)$. By adding $n-k$ rows and columns of $-\infty$ entries,
we obtain from $E$ an
idempotent matrix $F \in M_n(\trop)$ satisfying $C_\trop(F) \cong C_\trop(E) = C_\trop(A)$. But now
Theorem~\ref{thm_green} gives $F \GreenD A$, which suffices to establish (i).
\end{proof}


\begin{theorem}\label{thm_equalrank}
Every projective tropical polytope has generator
dimension equal to its dual dimension.
\end{theorem}
\begin{proof}
We show that the dual dimension cannot exceed the generator dimension, the
reverse inequality being dual by Propositions~\ref{prop_matrixdualrank} and
\ref{prop_regchar}. Suppose then
for a contradiction that $X$ is projective with dual dimension
$k$ strictly greater than its generator dimension $m$. Then by
Theorem~\ref{thm_dualrankembedding}, $k$ is minimal such that $X$ embeds
in $\ft^k$. But by Theorem~\ref{thm_projchar}, $X$ is isomorphic to
the column space of an $m \times m$ idempotent matrix $E$, say, which
means $X$ embeds in $\ft^m$. Since $m < k$ this is a contradiction.
\end{proof}

Note that Theorem~\ref{thm_projchar} says that every projective polytope is \textit{abstractly isomorphic} to the column space of an idempotent matrix (of size its generator dimension). Since a polytope is itself a submodule
of some $\ft^n$, we might ask whether every projective polytope is \textit{itself} the column space of an idempotent (of size the dimension of the containing space). We conclude this section by noting that \cite[Theorem 104]{Gaubert98} and \cite[Proposition 5]{Cohen06} extend to $\ft$, or in other words, projective polytopes are exactly the column spaces of idempotents:

\begin{theorem}\label{thm_projchar2}
Let $X \subseteq \ft^n$ be a tropical polytope. Then $X$ is projective if
and only if $X$ is the column space of an idempotent matrix in $M_n(\ft)$.
\end{theorem}
\begin{proof}
If $X$ is the column space of an idempotent over $\ft$ then
Theorem~\ref{thm_projchar} tells us that $X$ is projective.

Conversely, suppose $X$ is projective. By Theorem~\ref{thm_equalrank}, the
generator dimension of $X$ is equal to dual dimension of $X$, which by
Theorem~\ref{thm_dualrankembedding} cannot exceed $n$. Thus, we may write
$X$ as the column space of an $n \times n$ matrix $A$. By
Proposition~\ref{prop_regchar} this matrix is von Neumann regular as an
element of $M_n(\ft)$, so it
is $\GreenR$-related to an idempotent in $M_n(\ft)$, which by
Theorem~\ref{thm_green} also has column space $X$.
\end{proof}

\section{Order-Theoretic Properties of Projective Polytopes}\label{sec_order}

In this section we establish our order-theoretic characterisation of
projective tropical polytopes. For this, we first need some
elementary order-theoretic properties of idempotent matrices over the
tropical semiring. These will be familiar to experts but to aid the non-specialist
reader we include some short direct proofs.

\begin{proposition}\label{prop_linearorderpreserving}
For any matrix $A \in M_n(\trop)$ and vectors $x, y \in \trop^n$, if
$x \geq y$ then $Ax \geq Ay$ and $xA \geq yA$.
\end{proposition}
\begin{proof}
If $x \geq y$ then $x \oplus y = x$, so by linearity $Ax \oplus Ay = A(x \oplus y) = Ax$ which
means that $Ax \geq Ay$. The other claim is dual.
\end{proof}

The following result (and the corollaries that follow) can be seen as a special case of the
well-developed spectral theory for tropical matrices (see for example \cite[Theorem 3.101]{Baccelli92}).

\begin{lemma}\label{lemma_plentyzeros}
Let $E \in M_n(\trop)$ be an idempotent matrix, and let $x$
be an extremal point of the column space $C_\trop(E)$. Then there exists a
$\lambda \in \ft$ such that $\lambda x$ occurs as a column of $E$ with
$0$ in the diagonal position.
\end{lemma}
\begin{proof}
Clearly every extremal point of $C_\trop(E)$ occurs (up to scaling) as a
column of $E$, since they are by definition needed to generate the column
space. Let $c_1, \dots, c_n$ be the columns of $E$, and suppose $c_i$
is an extremal point.

Considering the equation $E^2 = E$, we have
$$c_i = \bigoplus_{j=1}^n E_{ji} c_j = \bigoplus_{j=1}^n (c_i)_j c_j.$$
Since $c_i$ is extremal, it must in fact be equal to one of the terms
in this sum, say
$$c_i = E_{ji} c_j = (c_{i})_j c_j,$$
giving that $c_j$ is a multiple of $c_i$.
Moreover, since $c_i$ is extremal, and hence not the zero vector, it follows that
$(c_i)_j \neq -\infty$. But now
$$(c_i)_j = ((c_i)_j c_j)_j = (c_i)_j (c_j)_j$$
which since $(c_i)_j \neq -\infty$ means that $(c_j)_j = E_{jj} = 0$.
Since $c_j$ is a multiple of $c_i$, this completes the proof.
\end{proof}

\begin{corollary}
If $E \in M_n(\trop)$ is an idempotent matrix of column generator
rank $n$ (or row generator rank $n$) then every diagonal entry of $E$ is $0$.
\end{corollary}

\begin{corollary}\label{cor_nonsingidptincreasing}
If $E \in M_n(\trop)$ is an idempotent matrix of column generator
rank $n$ (or row generator rank $n$) then $Ex \geq x$ and $xE \geq x$
for all $x \in \trop^n$.
\end{corollary}

\begin{proposition}\label{prop_nonsingidpt_minconvex}
If $E \in M_n(\trop)$ [respectively, $E \in M_n(\ft)$] is an idempotent
matrix of column generator rank
$n$ (or row generator rank $n$) then $C_\trop(E)$ and $R_\trop(E)$
[respectively, $C_\ft(E)$ and $R_\ft(E)$] are min-plus convex.
\end{proposition}
\begin{proof}
We prove the claim for $C_\trop(E)$, that for the row space
being dual and the $\ft$ cases very similar. Suppose $a, b \in C_\trop(E)$,
and let $c$ be the componentwise minimum of $a$ and $b$. It
will suffice to show that $c \in C_\trop(E)$. Since $a,b \in C_\trop(E)$ and
$E$ is idempotent, we have $a = Ea$ and $b = Eb$. Since $c \leq
a$ and $c \leq b$, by Proposition~\ref{prop_linearorderpreserving}
we have $Ec \leq E a = a$ and $Ec \leq Eb = b$. This means
that $Ec \leq \min\{a,b\} = c$. But by
Corollary~\ref{cor_nonsingidptincreasing} we have $Ec \geq c$, so
it must be that $Ec = c$ and $c \in C_\trop(E)$, as required.
\end{proof}

\begin{proposition}\label{prop_nonsingidpt_desc}
Let $E \in M_n(\trop)$ be an idempotent of column generator rank $n$.
Then
for any vector $x \in \trop^n$, the vector $Ex$ is the minimum (with respect
to the partial order $\leq$)
of all elements $y \in C_\trop(E)$ such that $y \geq x$.
\end{proposition}
\begin{proof}
By definition we have $Ex \in C_\trop(E)$ and by
Corollary~\ref{cor_nonsingidptincreasing} we have $Ex \geq x$, so
$Ex$ is itself an element of $C_\trop(E)$ which lies above $x$. Thus, it
will suffice to show that every other such element lies above $Ex$.
Suppose, then that $z \in C_\trop(E)$ and $z \geq x$. Since $z \in C_\trop(E)$
and $E$ is idempotent we have $Ez = z$. But since $z \geq x$,
Proposition~\ref{prop_linearorderpreserving} gives $z = Ez \geq Ex$,
as required.
\end{proof}

We note that Proposition~\ref{prop_nonsingidpt_minconvex} can also be
deduced as a consequence of Proposition~\ref{prop_nonsingidpt_desc}.
Proposition~\ref{prop_nonsingidpt_desc} has the following interesting
semigroup-theoretic corollary:
\begin{theorem}\label{thm_uniqueidpts}
Any $\GreenR$-class [$\GreenL$-class] in $M_n(\trop)$ consisting of matrices
of column generator rank $n$ [or row generator rank $n$] contains at most one
idempotent.
\end{theorem}
\begin{proof}
We prove the claim for $\GreenR$-classes, that for $\GreenL$-classes being
dual. Let $E$ be an idempotent such that $C_\trop(E)$ has generator rank $n$.
For $i \in \lbrace 1, \dots, n \rbrace$, applying Proposition~\ref{prop_nonsingidpt_desc} with $x = e_i$ the $i$th
standard basis vector shows that the $i$th column $E e_i$ of $E$ is
the minimum element of $C_\trop(E)$ greater than or equal to $e_i$. Thus, $E$ is
completely determined by its column space and the fact that it is idempotent.
Now if $F$ were another idempotent $\GreenR$-related to $E$ then by
Theorem~\ref{thm_green}(ii) we would have $C_\trop(E) = C_\trop(F)$, which by the
preceding argument would mean that $E = F$.
\end{proof}

Note the row or column generator rank hypothesis in Theorem~\ref{thm_uniqueidpts}
cannot be removed. Indeed, in \cite{K_tropicalgreen} it was shown that
every $\GreenH$-class corresponding to a $1$-generated column space in
$M_2(\trop)$ contains an idempotent, so the corresponding $\GreenR-$
and $\GreenL$-classes each contain a continuum of idempotents.

We are now ready to prove our first main result, namely Theorem~\ref{thm_order1}
from the introduction.

\begin{theorem14}
A tropical polytope in $\ft^n$ of generator dimension $n$ and dual dimension $n$
is a projective $\ft$-module if and only if it is min-plus (as well as max-plus) convex.
\end{theorem14}
\begin{proof}
The direct implication is immediate from Theorems~\ref{thm_projchar2} and
\ref{prop_nonsingidpt_minconvex}, so we need only prove the converse.

Suppose, then, that $X \subseteq \ft^n$ is min-plus and max-plus convex, and let $M$ be the matrix whose
$i$th column is the infimum (in $\ft^n$) of all elements $y \in X$
such that $y \geq e_i$, where $e_i$ is the $i$th standard basis vector (that
is, such that $y$ has non-negative $i$th coordinate). Such an infimum
exists. Indeed, if $n=1$ take $y=0$. Otherwise,
for each coordinate $j \neq i$,
consider the set
$$\lbrace u_j \mid u \in X \textrm{ with } u_i \geq 0 \rbrace.$$
It is easy to see that this set is non-empty and, since $X$ is finitely
generated, it has a
lower bound and
hence an infimum. It follows from the fact $X$ is closed that
this infimum will be attained; choose a vector $w_j \in X$ such that
$w_{jj}$ attains it at $w_{ji} \geq 0$. In fact, by the minimality of
$w_{jj}$ and the fact that $X$ is closed under scaling, we will have
$w_{ji} = 0$. Now let $v$ be the minimum of all the $w_j$'s. Then $v_i = 0$
and $v$ is clearly less than or equal to all vectors $u \in X$ with
$u_i \geq 0$. Moreover, by min-plus convexity, it lies in $X$, which
means it must be the desired minimum.

Notice that since $X$ is closed under scaling, it will have elements
in which the $i$th coordinate is $0$. It follows that the $i$th column of $M$
will in fact have $i$th coordinate $0$, that is, that every diagonal entry
of $M$ is $0$. We have shown that every column of $M$ lies in $X$, so
$C_\ft(M) \subseteq X$. We aim to show that $M$ is idempotent with column
space $X$.

First, we claim that each column of $M$ is an extremal point of $X$.
Indeed, suppose for a contradiction that the $p$th column, call it $y$,
is not an extremal point of $X$. Then by definition we may write $y$ as
a finite sum of elements in $X$ which are not multiples of $y$, say
$y = z_1 \oplus \dots \oplus z_k$. Let $j$ be such that $z_j$ agrees with
$y$ in the $p$th coordinate. Then $z_j < y$ (since $z_j$ forms part
of a linear combination for $y$, and was chosen not to be a multiple of $y$)
but $z_j \geq e_p$ and $z_j \in X$, which contradicts the choice of $y$
as the minimum element in $X$ above $e_p$.

Next, we claim that no two columns of $M$ are scalings of one another.
Indeed, suppose the $i$th column $v_i$ and $j$th column $v_j$ are scalings
of one another. For any $x \in X$ we have
$(-x_i) x \geq e_i$ and $(-x_i) x \in X$, so by the definition of $v_i$ we
have $v_i \leq (-x_i)x$. Thus, using the fact that $v_{ii} = 0$,
$$v_{ij} - v_{ii} \leq (-x_i) x_j = x_j - x_i.$$
By applying the same argument with $i$ and $j$ exchanged we also obtain
$$v_{ji} - v_{jj} \leq (-x_j) x_i = x_i - x_j.$$
But since $v_i$ is a multiple of $v_j$, we have $v_{ji} - v_{jj} = v_{ii} - v_{ij}$
so negating we get
$$v_{ij} - v_{ii} \geq x_j - x_i.$$
Thus we have shown that $x_j - x_i = v_{ij} - v_{ii}$ for every $x \in X$. In
other words, the $j$th entry of every vector in $X$ is determined by the $i$th
entry. This implies that $X$ embeds linearly into $\ft^{n-1}$, which
by Theorem~\ref{thm_dualrankembedding} contradicts the fact that $X$ has dual
dimension $n$.

We have shown that the $n$ columns of $M$ are extremal points of $X$, and that
no two are scalings of each other. Since $X$ has generator dimension $n$ it has
precisely $n$ extremal points up to scaling, so we conclude that every extremal
point of $X$ must occur (up to scaling) as a column of $M$. Thus,
$X \subseteq C_\ft(M)$, and so $C_\ft(M) = X$.

Finally, we need to show that $M$ is idempotent. We have already observed that
every diagonal entry of $M$ is $0$. It follows from the definition of matrix
multiplication that for all $i$ and $j$,
$$(M^2)_{ij} \geq M_{ij} M_{jj} = M_{ij} 0 = M_{ij}.$$
Now let $i, j, k \in \lbrace 1, \dots, n \rbrace$. To complete the proof, it
will suffice to show that $M_{ij} \geq M_{ik} M_{kj}$.
Let $v_j$ and $v_k$ be the $j$th and $k$th columns of $M$, and
consider the vector $w = (-M_{kj}) v_j = (-v_{jk}) v_j$. Then $w$ lies in $X$ and has $k$th component $0$, so
by the definition of $M$, $w$ is greater than $v_k$. In particular, comparing
the $i$th entries of these vectors, we have
$$(-M_{kj}) M_{ij} = w_i \geq v_{ki} = M_{ik}$$
and so
$$M_{ij} \geq M_{ik} M_{kj}$$
as required.
\end{proof}

Combining Theorem~\ref{thm_order1} with Proposition~\ref{prop_regchar}
yields an
order-theoretic characterisation of von Neumann regularity for
matrices of full column and row generator rank over $\ft$.

\begin{theorem}\label{thm_minmax}
Let $M \in M_n(\ft)$ be a matrix of column generator rank $n$
and row generator rank $n$. Then the following are equivalent:
\begin{itemize}
\item[(i)] $M$ is von Neumann regular;
\item[(ii)] $C_\ft(M)$ is min-plus convex;
\item[(iii)] $R_\ft(M)$ is min-plus convex.
\end{itemize}
\end{theorem}
\begin{proof}
By Proposition~\ref{prop_regchar}, $M$ is von Neumann regular if and
only if $C_\ft(M)$ is projective, which by Theorem~\ref{thm_order1} is
true exactly if $C_\ft(M)$ is min-plus convex. A similar argument applies
to $R_\ft(M)$.
\end{proof}

Next we prove Theorem~\ref{thm_order2} from the introduction.

\begin{theorem15}
A tropical polytope is projective if and only if it has
generator dimension equal to its dual dimension (equal to $k$, say), and is
linearly isomorphic to a submodule of $\ft^k$ that is min-plus convex (as
well as max-plus convex).
\end{theorem15}

\begin{proof}
Suppose $X \subseteq \ft^n$ is finitely generated and projective. By
Theorem~\ref{thm_equalrank}
it has generator dimension equal to its dual dimension; let $k$ be this value.
By Theorem~\ref{thm_projchar}, $X$ is isomorphic to the column space
of an idempotent matrix in $M_k(\ft)$. This column space is projective
and has generator dimension $k$ and dual dimension $k$, and so by
Theorem~\ref{thm_order1} is min-plus convex as well as max-plus
convex.

Conversely, if $X$ has dual dimension and generator dimension $k$ and is isomorphic
to a convex set in $\ft^k$ which is min-plus as well as max-plus convex,
then $X$ is projective by Theorem~\ref{thm_order1}.
\end{proof}

\section{Geometric characterisation of projective polytopes}\label{sec_geometry}

Our aim in this section is to prove Theorem~\ref{thm_geometric}, which gives a
geometric characterisation of projective tropical polytopes in terms of pure
dimension, generator dimension and dual dimension.

We require some preliminary terminology, notation
and results from \cite{Develin04}.
Let $v_1, \dots, v_k \in \ft^n$ and let $X \subseteq \ft^n$ be the polytope
they generate. Let $x \in \ft^n$. The \textit{type} of $x$ (with respect to
the vectors $v_1, \dots, v_k$) is an $n$-tuple of sets, the $p$th component
of which consists
of the indices of those generators which can contribute in the $p$th position
to a linear combination for $x$.
Formally,
\begin{align*}
\type(x)_p &= \lbrace i \in \lbrace 1, \dots, k \rbrace \mid \exists \lambda \in \ft \text{ such that } \lambda v_i \leq x \text{ and } \lambda v_{ip} = x_p \rbrace \\
&= \lbrace i \in \lbrace 1, \dots, k \rbrace \mid x_p - v_{ip} \leq x_q - v_{iq} \text{ for all } q \in \lbrace 1, \dots, n \rbrace \rbrace.
\end{align*}
It is easily seen that $X$ itself consists of those vectors whose types have every
component non-empty.

For a given type $S$ we write $S_p$ for the $p$th component of $S$. We denote
by $G_S$ the
undirected graph having vertices $\lbrace 1, \dots, n \rbrace$, and an edge
between $p$ and $q$ if
and only if $S_p \cap S_q \neq \emptyset$. We define union and inclusion for
types in the obvious way: if $S$ and $T$ are types then $S \cup T$ is the
type given by $(S \cup T)_i = S_i \cup T_i$, and $S \subseteq T$ if
$S_i \subseteq T_i$ for all $i$, that is, if $S \cup T = T$.
 We write $X_S$ for the set of all
points having type \textbf{containing} $S$; it is readily verified that $X_S$
is a closed set of pure dimension. A \textit{face} of $X$ is a set $X_S$ such
that $S$ is the type of a point of $X$.

We require the following result of Develin and
Sturmfels \cite{Develin04}, which we rephrase slightly
for compatibility with the
terminology and conventions of the present paper (\cite{Develin04} instead
using the min-plus convention and using the term ``dimension'' to mean projective
dimension).

\begin{lemma}[{\cite[Proposition 17]{Develin04}}]\label{lemma_graphdimension}
With notation as above, the tropical dimension of $X_S$ is the number
of connected components in $G_S$.
\end{lemma}

It is easily seen that a polytope has pure
dimension $k$ if and only if every point lies inside a closed face of
dimension $k$.

Let $E \in M_n(\ft)$ be an idempotent matrix with columns
$v_1, \dots, v_n$ (so that $v_{ij} = E_{ji}$ for all $i, j$). We
shall show that the column space $C_\ft(E)$ has pure dimension. To do
this we need some lemmas.

\begin{lemma}\label{lemma_relation}
Let $E, v_1, \dots, v_n$ be as above, and let $i, j \in \lbrace 1, \dots, n \rbrace$ be such that
\begin{itemize}
\item $v_i$ and $v_j$ are extremal points in $C_\ft(E)$;
\item $v_j$ is not a multiple of $v_i$; and
\item $v_{ii} = v_{jj} = 0$.
\end{itemize}
Then for every $k \in \lbrace 1, \dots n \rbrace$ we have
$$v_{ji} - v_{ii} = v_{ji} \leq v_{jk} - v_{ik},$$
and in the case $k = j$ this inequality is strict.
\end{lemma}
\begin{proof}
For any $k$, computing the $(k,j)$ entry of $E^2$, we see that
$$v_{jk} = E_{kj} = (E^2)_{kj} \geq E_{ki} E_{ij} = v_{ik} + v_{ji}$$
which, since $v_{ii} = 0$, yields
$$v_{ji} - v_{ii} = v_{ji} \leq v_{jk} - v_{ik}$$
as required.

Now let $k = j$, and suppose for a contradiction that the
inequality is not strict, that is, that
$$v_{ji} = v_{jk} - v_{ik} = v_{jj} - v_{ij}.$$
Since $v_{jj} = 0$, the above equation
yields $v_{ji} = - v_{ij}$. Now for any index $p \in \lbrace 1,
\dots, n \rbrace$ by the above we have
$$v_{jp} - v_{ip} \geq v_{ji}.$$
By symmetry of assumption, we may also apply a corresponding
inequality with $i$ and $j$ exchanged, which yields
$$v_{ip} - v_{jp} \geq v_{ij} = -v_{ji}$$
and hence by negating both sides
$$v_{jp} - v_{ip} \leq v_{ji}.$$
So $v_{ji} = v_{jp} - v_{ip}$ for all $p$, which means that
$v_j = v_{ji} v_i$. But this contradicts the
hypothesis that $v_j$ is not a scalar multiple of $v_i$.
\end{proof}

\begin{lemma}\label{lemma_everygenerator}
Let $E, v_1, \dots, v_n$ be as above, and let $J \subseteq \lbrace 1, \dots, n \rbrace$ be such that the
corresponding columns form a set of unique representatives for the
extremal points of $C_\ft(E)$, and $v_{jj} = 0$ for every $j \in J$.
Let $x \in C_\ft(E)$, and $j \in J$. Then $j \in \type(x)_j$.
\end{lemma}
\begin{proof}
Suppose for a contradiction that
$j \notin \type(x)_j$. Write
$$x = \bigoplus_{i \in J} \lambda_i v_i$$
with the $\lambda_i$ maximal. The fact that $j \notin \type(x)_j$
means precisely that
$\lambda_j v_{jj} < x_j$.

By definition of the sum, there is a $k \in J$ such that
$$\lambda_k v_{kj} = x_j > \lambda_j v_{jj}$$
and by the above $k \neq j$. Rearranging, we obtain
$$v_{kj} - v_{jj} > \lambda_j - \lambda_k.$$
On the other hand, by maximality of the $\lambda_i$'s, there is a
$p$ such that $\lambda_j v_{jp} = x_p$. Then certainly we have
$$\lambda_k v_{kp} \leq x_p = \lambda_j v_{jp}$$
which combining with the above yields
$$v_{kp} - v_{jp} \leq \lambda_j - \lambda_k < v_{kj} - v_{jj},$$
contradicting Lemma~\ref{lemma_relation} applied to columns $v_k$ and $v_j$.
\end{proof}

\begin{lemma}\label{lemma_middlepoint}
Let $E, v_1, \dots, v_n$ be as above and let $J \subseteq \lbrace 1, \dots, n \rbrace$ be such that the
corresponding columns form a set of unique representatives for the
extremal points of $C_\ft(E)$, and $v_{jj} = 0$ for all $j \in J$. Let
$x \in C_\ft(E)$. Then there is an element $y \in C_\ft(E)$ such that
$\type(y) \subseteq \type(x)$ and $\type(y)$ is a vector of
singletons.
\end{lemma}
\begin{proof}
For every vector $y \in \ft^n$, define
$$T_y = \lbrace (i,j,p) \in J \times J \times \lbrace 1, \dots, n \rbrace \mid i \neq j \text{ and } i,j \in \type(y)_p \rbrace.$$
If $y \in C_\ft(E)$ then, as discussed above, the components of $\type(y)$
are non-empty, so $\type(y)$ is a vector of singletons exactly
if $T_y$ is empty. Thus, the claim to be proven is that there is a
vector $y \in C_\ft(E)$ with $\type(y) \subseteq \type(x)$ and $T_y$ empty.
Suppose false for a contradiction, and choose $z \in C_\ft(E)$ such
that $\type(z) \subseteq \type(x)$, and the cardinality of $T_z$ is
minimal amongst vectors having this property.

Write
$$z = \bigoplus_{i \in J} \lambda_i v_i$$
with the $\lambda_i$'s maximal. By the supposition, $T_z$ is
non-empty, so we may choose some $(i,j,p) \in T_z$. Then by the
definition of types we have
$$\lambda_i v_{ip} = z_p = \lambda_j v_{jp}.$$

Notice that we cannot have both
$$\lambda_i v_{ij} = z_j \text{ and } \lambda_j v_{ji} = z_i.$$
Indeed, by Lemma~\ref{lemma_everygenerator} we have $\lambda_i v_{ii}
= z_i$ and $\lambda_j v_{jj} = z_j$, so we would have
$$v_{jj} - v_{ij} = \lambda_i - \lambda_j = v_{ji} - v_{ii}$$
which contradicts the strict inequality guaranteed by
Lemma~\ref{lemma_relation}. Thus, by exchanging $i$ and $j$ if
necessary, we may assume without loss of generality that
$\lambda_i v_{ij} < z_j$.

Now choose an $\varepsilon > 0$ such that
$$\varepsilon < z_q - \lambda_k v_{kq}$$
for all $k \in J$ and $q \in \lbrace 1, \dots, n \rbrace$ such
that $\lambda_k v_{kq} \neq z_q$. (Notice that by the definition
of the $\lambda_k$ we can never have $z_q < \lambda_k v_{kq}$, so
the condition $\lambda_k v_{kq} \neq z_q$ is sufficient to make
$z_q - \lambda_k v_{kq}$ positive.)

Define
$$y = z \oplus (\lambda_i + \varepsilon) v_i.$$
Since $z \in C_\ft(E)$ and $v_i$ is a column of $E$, we have $y \in C_\ft(E)$. Write
$$y = \bigoplus_{k \in J} \mu_k v_k$$
with the $\mu_k$'s maximal. Notice that since $\lambda_i v_i \leq z$, we have
$$y = z \oplus (\lambda_i + \varepsilon) v_i = z \oplus \varepsilon (\lambda_i v_i) \leq \varepsilon z.$$
In other words, no coordinate of $y$ can exceed the corresponding
coordinate in $z$ by more than $\varepsilon$. It follows immediately
that
$$\mu_k \leq \lambda_k + \varepsilon$$
for all $k \in J$. It is also clear that $\mu_i = \lambda_i + \varepsilon$.

We claim that $\type(y) \subseteq \type(z)$. Indeed, suppose
$k \notin \type(z)_p$, that is, $\lambda_k v_{kp} < z_p$. Then by
the choice of $\varepsilon$, we have $\varepsilon < z_p - \lambda_k
v_{kp}$, that is, $\varepsilon \lambda_k v_{kp} < z_p$. So using
the previous paragraph we have
$$\mu_k v_{kp} \leq \varepsilon \lambda_k v_{kp} < z_p \leq y_p,$$
which means that $k \notin \type(y)_p$.

It follows immediately that $T_y \subseteq T_z$; we claim that the
containment is strict. Indeed, we know that $(i,j,p) \in T_z$;
suppose for a contradiction that it lies also in $T_y$, that is,
that $i, j \in \type(y)_p$. Then by definition we have
$$\mu_j v_{jp} = y_p = \mu_i v_{ip} = \varepsilon \lambda_i v_{ip} = \varepsilon \lambda_j v_{jp},$$
from which we deduce that $\mu_j = \lambda_j + \varepsilon$. Thus, using
Lemma~\ref{lemma_everygenerator}, we have
$$y_j = \mu_j v_{jj} = \varepsilon \lambda_j v_{jj} = \varepsilon z_j > z_j.$$
Since $y = z \oplus \varepsilon \lambda_i v_i$, the only way this can
happen is if
$$y_j = (\varepsilon \lambda_i v_i)_j = \varepsilon \lambda_i v_{ij}.$$
But then $\varepsilon \lambda_i v_{ij} = y_j = \varepsilon \lambda_j
v_{jj}$, so
$$\lambda_i v_{ij} = \lambda_j v_{jj} = z_j.$$
This contradicts our assumption that $\lambda_i v_{ij} < z_j$, and
so proves the claim that $(i,j,p) \notin T_y$. Thus, $T_y$ is a
strict subset of $T_z$, which contradicts the minimality
assumption on $T_z$, and completes the proof of the lemma.
\end{proof}

\begin{theorem}\label{thm_projimppure}
Let $E \in M_n(\ft)$ be an idempotent matrix of column generator
rank $r$. Then $C_\ft(E)$ has pure dimension $r$.
\end{theorem}

\begin{proof}
Let $x \in C_\ft(E)$. It will suffice to show that $x$ lies in a face of
tropical dimension $r$.

Let $J \subseteq \lbrace 1, \dots, n \rbrace$ be such that the
corresponding columns form a set of unique representatives for the
extremal points of $C_\ft(E)$. Thus, $J$ has cardinality $r$. By Lemma~\ref{lemma_plentyzeros}, we
may choose $J$ so that $v_{jj} = 0$ for all $j \in J$. Now
consider types with respect to the generating set of $C_\ft(E)$ formed
by the columns corresponding to indices in $J$.

By Lemma~\ref{lemma_middlepoint} there is a point $y \in C_\ft(E)$
such that $\type(y) \subseteq \type(x)$ and $\type(y)$ is a vector of
singletons. By Lemma~\ref{lemma_everygenerator}, we have
$j \in \type(y)_j$ for every $j \in J$. It follows that the graph $G_{\type(y)}$
has exactly $r$ connected components (one corresponding to each generator
$v_i$).

Hence, by Lemma~\ref{lemma_graphdimension}, the face
$X_{\type(y)}$ has tropical dimension $r$. Moreover, since
$\type(y) \subseteq \type(x)$, it
follows from the definition of $X_{\type(y)}$ that $x$ lies in a face of
tropical dimension $r$, as required.
\end{proof}

\begin{lemma}\label{lemma_uniqueface}
Let $X$ be a tropical polytope in $\ft^n$ of generator dimension $n$ or
less. Then $X$ contains at most one face of dimension $n$.
\end{lemma}
\begin{proof}
Choose generators $v_1, \dots, v_n$ for $X$, and suppose for a contradiction
that $X_S$ and $X_T$ are distinct faces of dimension $n$.
By Lemma~\ref{lemma_graphdimension}, both $S$ and $T$
are $n$-tuples of singleton sets containing all the numbers from $1$ to $n$.
By reordering our chosen generating set if necessary we may thus assume that
$$S = (\lbrace 1 \rbrace, \lbrace 2 \rbrace, \dots, \lbrace n \rbrace)$$
while
$$T = (\lbrace \sigma(1) \rbrace, \lbrace \sigma(2) \rbrace, \dots, \lbrace \sigma(n) \rbrace)$$
for some permutation $\sigma$ of $\lbrace 1, \dots, n \rbrace$. Since $S$
and $T$ are distinct, $\sigma$ must be non-trivial.

Now choose points $a, b \in X$ with $\type(a) = S$ and $\type(b) = T$. Write
$$a = \bigoplus_{i=1}^n \lambda_i v_i \textrm{ and } b = \bigoplus_{i=1}^n \mu_i v_i$$
with the $\lambda_i$'s and $\mu_i$'s all maximal. By the definition of types,
for all $i$ we have
$$a_i = \lambda_i v_{ii} \geq \lambda_{\sigma(i)} v_{\sigma(i) i} \text{ and } \mu_i v_{ii} \leq \mu_{\sigma(i)} v_{\sigma(i) i} = b_i$$
and these inequalities are strict provided $i \neq \sigma(i)$. Rearranging
these, we get
$$\lambda_i - \lambda_{\sigma(i)} \geq v_{\sigma(i) i} - v_{ii} \geq \mu_i - \mu_{\sigma(i)}$$
and again, the inequalities are strict provided $i \neq \sigma(i)$.

Now since $\sigma$ is a non-trivial permutation of a finite set, it contains a non-trivial
cycle. In other words, there is a $p \in \lbrace 1, \dots, n \rbrace$ and
an integer $k \geq 2$ such
that $p \neq \sigma(p)$, but $p = \sigma^k(p)$. Note that, $\sigma^i(p) \neq \sigma^{i+1}(p)$ for
any $i$, so using our strict inequalities above we have
$$0 = \sum_{i = 1}^k (\lambda_{\sigma^i(p)} - \lambda_{\sigma^{i+1}(p)})
 > \sum_{i = 1}^k (\mu_{\sigma^i(p)} - \mu_{\sigma^{i+1}(p)}) = 0$$
giving the required contradiction.
\end{proof}

\begin{theorem}\label{thm_pureimpproj}
Suppose $X \subseteq \ft^n$ is a tropical polytope of generator dimension
$n$ and pure dimension $n$. Then $X$ is projective.
\end{theorem}

\begin{proof}
Let $u_1, \dots, u_n$ be a minimal generating set for $X$ (so that the
elements $u_i$ are unique representatives of the extremal points of $X$). Consider
the types of points in $X$ with respect to this generating set. Since $X$
has pure dimension, for each $i$, $u_i$ is contained in a closed face of
dimension $n$. It follows by Lemma~\ref{lemma_uniqueface} that all of the
$u_i$'s are contained in the \textbf{same} face of dimension $n$, say $X_S$
for some type $S$. Since $X_S$ is a face of $X$, the components of $S$ are
non-empty, so it follows by Lemma~\ref{lemma_graphdimension} that
$S$ consists of singletons and contains every $i$. By reordering the
$u_i$'s, we may assume without loss of generality that $S_{k} = \lbrace k \rbrace$
for all $k \in \lbrace 1, \dots, n \rbrace$.

Moreover, by scaling the $u_i$'s if necessary, we may assume that
$u_{ii} = 0$ for each $i$. Let $E \in M_n(\ft)$
be the matrix whose $i$th column is $u_i$. It is immediate that
$C_\ft(E) = X$, and from our rescaling of the $u_i$'s that the diagonal
entries of $E$ are $0$. We claim that $E$ is idempotent, that is,
$$(E^2)_{ij} = E_{ij}$$
for all $i, j \in \lbrace 1, \dots, n \rbrace$.

Fix $i, j \in \lbrace 1, \dots, n \rbrace$. From the definition of matrix
multiplication we have
$$(E^2)_{ij} = \bigoplus_{k=1}^n E_{ik} E_{kj}.$$
Since the diagonal entries of $E$ are $0$, we have
$$(E^2)_{ij} \ \geq \ E_{ii} E_{ij} \ = \ 0 E_{ij} \ = \ E_{ij}.$$

On the other hand, let
$k \in \lbrace 1, \dots, n \rbrace$. Recall that
$S_k = \lbrace k \rbrace$.
Since $u_j$
appears in the face $X_S$, by definition we have
$S \subseteq \type(u_j)$, and so $k \in \type(u_j)_k$. It follows
from the definition of types that
$$u_{jk} - u_{kk} \leq u_{ji} - u_{ki}.$$
But $u_{kk} = 0$ so rearranging yields $u_{ki} + u_{jk} \leq u_{ji}$
for all $k$. Thus
we have
$$(E^2)_{ij} \ = \ \bigoplus_{k=1}^n E_{ik} E_{kj} \ = \ \bigoplus_{k=1}^n u_{ki} + u_{jk} \ \leq \ \bigoplus_{k=1}^n u_{ji} \ = \ u_{ji} \ = \ E_{ij}.$$
as required to complete the proof of the claim that $E$ is idempotent.

Thus, $X$ is the column space of an idempotent matrix, so by
Theorem~\ref{thm_projchar2} we deduce that $X$ is projective.
\end{proof}

We are finally ready to prove Theorem~\ref{thm_geometric} from the introduction.

\begin{theorem11}
Let $X \subseteq \ft^n$ be a tropical polytope. Then $X$ is a projective $\ft$-module
if and only if it has pure dimension equal to its generator dimension and
its dual dimension.
\end{theorem11}
\begin{proof}
Suppose $X \subseteq \ft^n$ is a projective polytope. Then by Theorem~\ref{thm_equalrank},
there is a $k \leq n$ such that $X$ has generator dimension $k$ and dual dimension $k$.
Now by Theorem~\ref{thm_projchar}, $X$ is isomorphic to the column
space $C_\ft(E)$ of an idempotent matrix $E \in M_k(\ft)$.
It follows by Theorem~\ref{thm_projimppure} that $C_\ft(E)$ has pure dimension $k$.
Moreover, it is easy to see that a linear isomorphism of convex sets is a
homeomorphism with respect to the standard product topology inherited from
the real numbers. Indeed, an isomorphism is a bijection, and both it and its
inverse are continuous because addition and multiplication in $\ft$ are continuous. Since pure dimension is an
abstract topological property it follows that $X$ has pure dimension $k$.

Conversely, suppose $X$ has pure dimension, generator dimension and dual
dimension all equal to $k$. Then by Theorem~\ref{thm_dualrankembedding}, $X$ is isomorphic
to a convex set $Y \subseteq \ft^k$. Now $Y$ also has generator dimension $k$
and, by the same argument as above, pure dimension $k$ so by Theorem~\ref{thm_pureimpproj},
$Y$ is projective, and so $X$ is projective.
\end{proof}

\section{Examples}\label{sec_examples}

In this section we collect together some examples of tropical polytopes
in low dimension, and show how the concepts and results of this paper
apply to them.

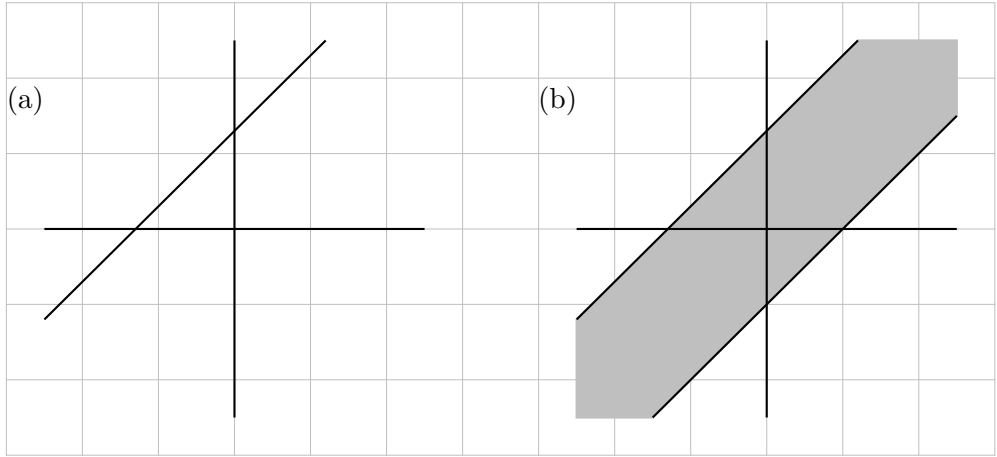
\begin{figure}[h]
\begin{pspicture}(13,6)
\psgrid[gridcolor=lightgray, gridwidth=0.25pt, gridlabels=0pt, subgridwidth=0.25pt, subgriddiv=1](0,0)(13,6)
\rput(0.25,4.7){(a)}
\qline(0.5,3)(5.5,3)
\qline(3,0.5)(3,5.5)
\qline(0.5,1.8)(4.2,5.5)
\rput(7.25,4.7){(b)}
\pspolygon[fillstyle=solid, fillcolor=lightgray, linecolor=lightgray](7.5,1.8)(7.5,0.5)(8.5,0.5)(12.5,4.5)(12.5,5.5)(11.2,5.5)
\qline(7.5,3)(12.5,3)
\qline(10,0.5)(10,5.5)
\qline(7.5,1.8)(11.2,5.5)
\qline(8.5,0.5)(12.5,4.5)
\end{pspicture}
\caption{Polytopes in $\ft^{2}$.}
\label{fig_twodimensions}
\end{figure}

We consider first the (somewhat degenerate) $2$-dimensional case.
It is well known and easily seen that every polytope in $\ft^2$ is
either (a) a line of gradient $1$, or (b) the closed region between two such
lines. Figure~\ref{fig_twodimensions} illustrates these possibilities. It
is readily verified that polytopes of type (a) have pure
dimension, generator dimension and dual dimension all equal to $1$,
while those of type (b) have pure dimension, generator dimension and
dual dimension all equal
to $2$. We deduce by Theorem~\ref{thm_geometric} that every tropical polytope in
$\ft^2$ is projective. By
Corollary~\ref{cor_regular}, we recover the fact (proved by explicit
computation in \cite{K_tropicalgreen}) that every $2 \times 2$ tropical
matrix is von Neumann regular, that is, that the semigroup of all such
matrices is a regular semigroup. It also follows by
Corollary~\ref{cor_equalrank} that the various notions of rank discussed
in the introduction all coincide for $2 \times 2$ matrices.

Recall that from affine tropical $n$-space we obtain projective tropical
$(n-1)$-space, denoted $\mathbb{PFT}^{(n-1)}$, by identifying two vectors
if one is a tropical multiple of the other by an element of $\mathbb{FT}$.
Thus we may identify $\pft^{n-1}$ with $\mathbb{R}^{n-1}$ via the map
\begin{equation}
(x_1, \dots, x_n) \mapsto (x_1 - x_n, x_2 - x_n, \dots, x_{n-1} - x_n).
\end{equation}
Each convex set $X \subseteq \ft^n$ induces a subset of the corresponding
projective space, termed the \emph{projectivisation} of $X$.

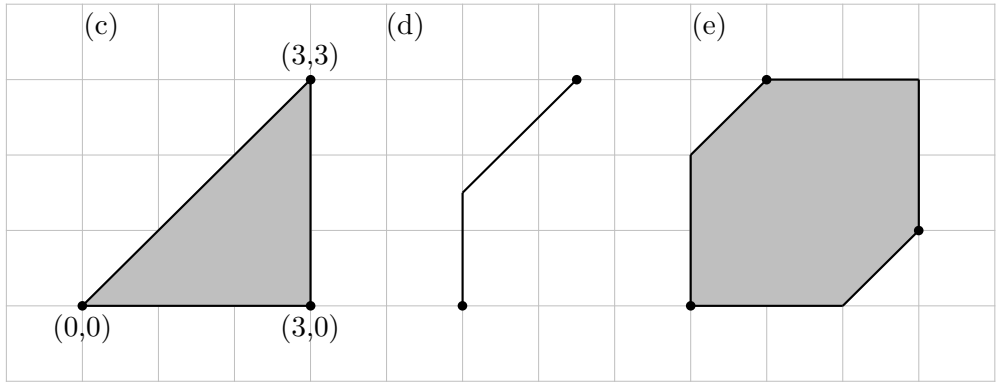
\begin{figure}[h]
\begin{pspicture}(13,5)
\psgrid[gridcolor=lightgray, gridwidth=0.25pt, gridlabels=0pt, subgridwidth=0.25pt, subgriddiv=1](0,0)(13,5)
\rput(1.25,4.7){(c)}
\pspolygon[fillstyle=solid, fillcolor=lightgray, linecolor=lightgray](1,1)(4,1)(4,4)
\qline(1,1)(4,1)
\qline(1,1)(4,4)
\qline(4,1)(4,4)
\psdots*(1,1)(4,1)(4,4)
\rput(1,0.7){(0,0)}
\rput(4,0.7){(3,0)}
\rput(4,4.3){(3,3)}
\rput(5.25,4.7){(d)}
\qline(6,2.5)(7.5,4)
\qline(6,2.5)(6,1)
\psdots*(7.5,4)(6,1)
\rput(9.25,4.7){(e)}
\pspolygon[fillstyle=solid, fillcolor=lightgray, linecolor=lightgray](9,3)(10,4)(12,4)(12,2)(11,1)(9,1)
\qline(9,1)(9,3)
\qline(9,3)(10,4)
\qline(10,4)(12,4)
\qline(12,4)(12,2)
\qline(12,2)(11,1)
\qline(11,1)(9,1)
\psdots*(10,4)(12,2)(9,1)
\end{pspicture}
\caption{Some projective tropical polytopes in $\pft^{2}$.}
\label{fig_projective}
\end{figure}

All three polytopes shown in Figure~\ref{fig_projective} have pure dimension.
Polytopes (c) and (e) have tropical dimension $3$, generator dimension
$3$ and dual dimension $3$, while polytope (d) has tropical dimension
$2$, generator dimension $2$ and dual dimension $2$, and so by
Theorem~\ref{thm_geometric} they are all projective. By Theorem~\ref{thm_order1}
polytopes (c) and (e) must be min-plus (as well as max-plus) convex, and
indeed this can be verified by inspection. Polytope (d) is not min-plus
convex, but by Theorem~\ref{thm_order2} it must be isomorphic to a polytope
in $\ft^2$ which \textit{is} min-plus convex; in fact it will be isomorphic
to something of the form shown in Figure~\ref{fig_twodimensions}(b).

By Corollary 1.2 we deduce that every matrix whose row space is one of
these polytopes must be von Neumann regular, and so there is at least
one idempotent matrix with each of these row spaces. In cases (c) and (e),
Theorem~\ref{thm_uniqueidpts} tells us that there is a unique such idempotent.
In case (d) Theorem~\ref{thm_uniqueidpts} does not apply (since the dimension is not maximal)
and in fact there are continuum-many such idempotents. The unique idempotent
in case (c) is
$$\left(\begin{array}{c c c}
0& -3& -3\\
0& 0&-3\\
0&0&0\end{array}\right).$$

\begin{figure}[h]
\begin{pspicture}(13,5)
\psgrid[gridcolor=lightgray, gridwidth=0.25pt, gridlabels=0pt, subgridwidth=0.25pt, subgriddiv=1](0,0)(13,5)
\rput(1.25,4.7){(f)}
\pspolygon[fillstyle=solid, fillcolor=lightgray, linecolor=lightgray](1,1)(3.5,1)(3.5,3.5)
\qline(1,1)(3.5,1)
\qline(1,1)(3.5,3.5)
\qline(3.5,0)(3.5,3.5)
\psdots*(1,1)(3.5,0)(3.5,3.5)
\rput(5.25,4.7){(g)}
\qline(6,2.5)(7.5,4)
\qline(6,2.5)(6,1)
\qline(5,2.5)(6,2.5)
\psdots*(5,2.5)(7.5,4)(6,1)
\rput(9.75,4.7){(h)}
\psframe[fillstyle=solid, fillcolor=lightgray, linecolor=lightgray](9,2)(12,4)
\psframe[fillstyle=solid, fillcolor=lightgray, linecolor=lightgray](10.5,1)(12,4)
\qline(9,2)(9,4)
\qline(9,4)(12,4)
\qline(12,4)(12,1)
\qline(12,1)(10.5,1)
\qline(10.5,1)(10.5,2)
\qline(10.5,2)(9,2)
\psdots*(9,2)(9,4)(10.5,1)(12,1)
\end{pspicture}
\caption{Some non-projective tropical polytopes in $\pft^{2}$.}
\label{fig_nonprojective}
\end{figure}
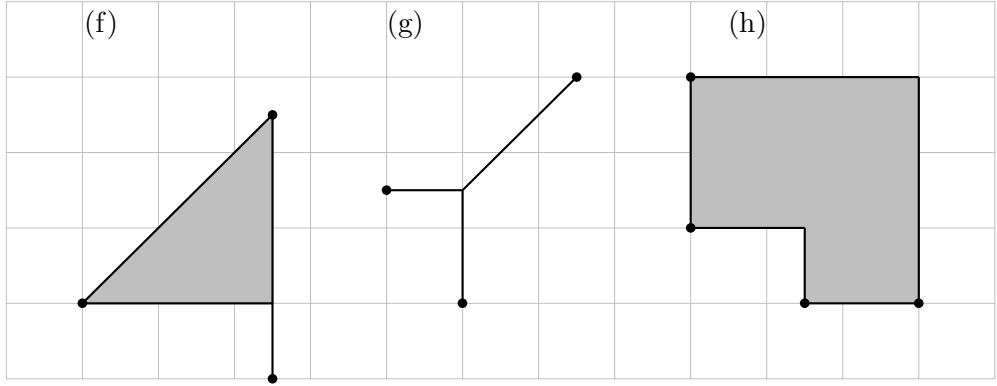

Figure~\ref{fig_nonprojective} illustrates three polytopes in $\ft^3$ which
fail to be projective for different reasons. Polytope (f) does not have
pure dimension, and so by Theorem~\ref{thm_geometric} cannot be projective.
Since the generator dimension and dual dimension are both equal to the
dimension of the ambient space, we may also deduce this from
Theorem~\ref{thm_order1} and the fact it is not min-plus convex.

Polytope (g), which is equivalent to \cite[Example
18]{Cohen06}, does have pure dimension, but its tropical dimension ($2$)
differs from its generator dimension and dual dimension (both $3$), and
hence by Theorem~\ref{thm_geometric} is not projective. Again, since
the generator dimension and dual dimension are both equal to the dimension
of the ambient space, non-projectivity also follows from
Theorem~\ref{thm_order1} and the lack of min-plus convexity.

Polytope (h) has pure dimension, but this time its tropical and dual
dimension ($3$) fail to agree with its generator dimension ($4$), so
again by Theorem~\ref{thm_geometric} it is not projective. In this case
Theorem~\ref{thm_order1} does not apply.
Note that if we choose a $4 \times 3$ matrix with row space polytope (h),
the \textit{column} space of this matrix will (by Proposition~\ref{prop_matrixdualrank})
yield an example of a polytope in $\ft^4$ with pure tropical dimension $3$,
generator dimension $3$ and dual dimension $4$.
\section*{Acknowledgement}
The authors thank the anonymous referee for their many useful comments, and in particular for drawing our attention to several related results in the literature.

\bibliographystyle{plain}

\end{document}